\documentclass[11pt,final]{article}

\usepackage{textcomp}
\usepackage{amsmath,amsthm,amssymb,amscd}
\usepackage{color}

\numberwithin{equation}{section}

\theoremstyle{definition}\newtheorem{definition}{Definition}
\theoremstyle{plain}\newtheorem{theorem}[definition]{Theorem}
\theoremstyle{plain}\newtheorem{proposition}[definition]{Proposition}
\theoremstyle{plain}\newtheorem{corollary}[definition]{Corollary}
\theoremstyle{plain}\newtheorem{lemma}[definition]{Lemma}
\theoremstyle{definition}
\theoremstyle{definition}\newtheorem{condition}[definition]{Condition}
\theoremstyle{definition}
\theoremstyle{definition}

\newcommand{\cO}{{\mathcal O}}
\newcommand{\cR}{{\mathcal R}}
\newcommand{\cN}{{\mathcal N}}

\DeclareMathOperator*{\sgn}{sgn}

\newcommand{\bbN}{\mathbb{N}}
\newcommand{\la}{\langle}
\newcommand{\ra}{\rangle}

\newcommand{\N}{\mathbb{N}}

\begin{document}

\title{Injectivity and weak*-to-weak continuity suffice for convergence rates in $\ell^1$-regularization}

\author{
\textsc{Jens Flemming} and \textsc{Daniel Gerth}
\footnote{Chemnitz University of Technology,
Faculty of Mathematics, D-09107 Chemnitz, Germany,
jens.flemming@mathematik.tu-chemnitz.de, daniel.gerth@mathematik.tu-chemnitz.de.}
}

\date{\today\\~\\
\small\textbf{Key words:} linear ill-posed problem, sparsity promoting regularization, Tikhonov regularization, source condition, variational source condition, convergence rates\\
~\\
\textbf{MSC:} 65J20, 47A52}

\maketitle

\begin{abstract}
We show that the convergence rate of $\ell^1$-regularization for linear ill-posed equations is always $\cO(\delta)$ if the exact solution is sparse and if the considered operator is injective and weak*-to-weak continuous.
Under the same assumptions convergence rates in case of non-sparse solutions are proven.
The results base on the fact that certain source-type conditions used in the literature for proving convergence rates are automatically satisfied.
\end{abstract}

\section{Setting and main theorem}

Let $A:\ell^1\rightarrow Y$ be a bounded linear operator mapping absolutely summable real sequences into a real Banach space $Y$.
For solving the possibly ill-posed equation
\begin{equation}\label{eq:problem}
Ax=y^\dagger,\quad x\in\ell^1,
\end{equation}
we consider $\ell^1$-regularization.
That is, given noisy data $y^\delta$ in $Y$ with
\begin{equation*}
\|y^\delta-y^\dagger\|_Y\leq\delta
\end{equation*}
for some positive $\delta$, we solve
\begin{equation}\label{eq:tikh}
T_\alpha^\delta(x):=\|Ax-y^\delta\|_Y^p+\alpha\|x\|_{\ell^1}\to\min_{x\in\ell^1}.
\end{equation}
Here $\alpha>0$ is the regularization parameter controlling the influence of the penalty term and $p>1$ is some exponent which can be used to simplify numerical minimization.

By $x_\alpha^\delta$ we denote a minimizer of $T_\alpha^\delta$. Throughout this article we assume that \eqref{eq:problem} has a solution $x^\dagger$ in $\ell^1$ and the aim is to find asymptotic estimates (convergence rates) for the solution error $\|x_\alpha^\delta-x^\dagger\|_{\ell^1}$ in terms of the noise level $\delta$.
To ensure convergence of $x_\alpha^\delta$ to $x^\dagger$ we have to choose $\alpha$ in the right way depending on $\delta$ and $y^\delta$. In the following we restrict our attention to \emph{a priori} choices $\alpha=\alpha(\delta)$ and to the \emph{discrepancy principle}
\begin{equation}\label{eq:dp}
\delta\leq\|Ax_\alpha^\delta-y^\delta\|_Y\leq\tau\delta
\end{equation}
with $\tau\geq 1$. The later means that $\alpha=\alpha(\delta,y^\delta)$ is chosen such that the corresponding discrepancy is close to $\delta$. Both parameter choice methods are well-known and we refer to \cite{EngHanNeu96,SchKalHofKaz12} for details.

Since $\ell^1$ is the dual space of $c_0$, the space of sequences converging to zero, we have the notion of weak* convergence in $\ell^1$ at hand. If $A$ is sequentially weak*-to-weak continuous, then $T_\alpha^\delta$ has minimizers and $\ell^1$-regularization \eqref{eq:tikh} is a stable and convergent method. In addition, the minimizers are sparse, that is they have only finitely many non-zero components. These results can be found, e.\,g., in \cite[Section~2]{Fle16}.

Convergence rates for $\ell^1$-regularization in infinite dimensions were obtained at first in \cite[Proposition~4.7]{DauDefDem04} based on smoothing properties of $A$. In \cite{BreLor09} rates were obtained if the canonical basis of $\ell^1$ belongs to the range of the adjoint $A^\ast$ and in \cite{Gra10c} rates were shown for more general penalties in the Tikhonov functional based only on an injectivity-type assumption, but $\ell^q$-regularization is only covered if $q<1$. Further rates results can be found in \cite{GraHalSch11,Lor08} based on a Banach space source condition.
The mentioned convergence rates results only hold if the exact solution $x^\dagger$ is sparse.
First rates results for non-sparse solutions were presented in \cite{BurFleHof13} with the same range condition for the canonical basis as in \cite{BreLor09}. Under weaker assumptions same results were proven in \cite{FleHeg14,FleHofVes16}.

Now we state our main result which shows that next to injectivity and sequential weak*-to-weak continuity of $A$ no further assumptions like source conditions are needed to prove convergence rates for sparse (see corollary below) and non-sparse solutions.
In fact, sequential weak*-to-weak continuity does not restrict the scope of application because to our best knowledge this property is the weakest assumption ensuring existence of regularized solutions as well as stability and convergence of $\ell^1$-regularization (cf.\ \cite{Fle16,SchKalHofKaz12}).
Thus, the only essential assumption is injectivity of $A$.

\begin{theorem}\label{th:main}
Let $A:\ell^1\rightarrow Y$ be an injective and sequentially weak*-to-weak continuous bounded linear operator and denote by $x^\dagger\in\ell^1$ the solution of \eqref{eq:problem}.
Then there are a continuous, concave and monotonically increasing function $\varphi:[0,\infty)\rightarrow[0,\infty)$ with $\varphi(0)=0$ and a constant $c$ such that
\begin{equation*}
\|x_\alpha^\delta-x^\dagger\|_{\ell^1}\leq c\varphi(\delta)\qquad\text{for all }\delta>0,
\end{equation*}
if the regularization parameter $\alpha$ is chosen a priori $\alpha\sim\frac{\delta^p}{\varphi(\delta)}$ or by the discrepancy principle \eqref{eq:dp}.
Further, there is always a monotonically increasing sequence $(\gamma_n)_{n\in\bbN}$ of positive numbers such that $\varphi$ can be chosen
\begin{equation}\label{eq:phi}
\varphi(t):=2\inf_{n\in\bbN}\left(\sum_{k=n+1}^\infty\vert x^\dagger_k\vert+\gamma_n t\right).
\end{equation}
The constant $c$ is independent of $x^\dagger$ and $\varphi$.
\end{theorem}

The proof will be given in Sections~\ref{sc:proof} and \ref{sc:proof2}, where also the constant $c$ is made explicit.
In case of sparse solutions the theorem specializes to the following result.

\begin{corollary}
Let $A:\ell^1\rightarrow Y$ be an injective and sequentially weak*-to-weak continuous bounded linear operator and denote by $x^\dagger\in\ell^1$ the solution of \eqref{eq:problem}.
If $x^\dagger$ is sparse, then there is a constant $c$ such that
\begin{equation*}
\|x_\alpha^\delta-x^\dagger\|_{\ell^1}\leq c\delta\qquad\text{for all }\delta>0,
\end{equation*}
if the regularization parameter $\alpha$ is chosen a priori $\alpha\sim\delta^{p-1}$ or by the discrepancy principle \eqref{eq:dp}.
The constant $c$ depends on the number of non-vanishing components of $x^\dagger$.
\end{corollary}

\begin{proof}
Let $x^\dagger_k=0$ for $k>m$.
Then $\varphi$ in Theorem~\ref{th:main} satisfies
\begin{equation*}
\varphi(t)\leq 2\sum_{k=m+1}^\infty\vert x^\dagger_k\vert+2\gamma_m t=2\gamma_m t
\end{equation*}
and the corresponding error estimate reduces to
\begin{equation*}
\|x_\alpha^\delta-x^\dagger\|_{\ell^1}\leq 2c\gamma_m\delta.
\end{equation*}
\end{proof}

In the next section we discuss the relation of the range of $A^\ast$ to smoothness properties of basis elements. Then we go on to the proof of our main theorem in Sections~\ref{sc:proof} and \ref{sc:proof2}.

\section{The range of $A^\ast$ and basis smoothness}\label{sc:basis}

Denote by $\ell^\infty=(\ell^1)^\ast$ the space of bounded sequences and by $A^\ast:Y^\ast\rightarrow\ell^\infty$ the adjoint of $A$.
The range $\cR(A^\ast)$ of $A^\ast$ played a crucial role in several results on convergence rates for $\ell^1$-regularization and in the present section we discuss some known results on its structure from the literature.

Denoting by $(e^{(k)})_{k\in\bbN}$ the canonical basis of $\ell^1$, the following lemma has been proven in \cite{Fle16}.

\begin{lemma}
The following assertions are equivalent:
\begin{itemize}
\item[(i)]
$A$ is sequentially weak*-to-weak continuous.
\item[(ii)]
$\cR(A^\ast)\subseteq c_0$.
\item[(iii)]
$(A\,e^{(k)})_{k\in\bbN}$ converges weakly to zero.
\end{itemize}
\end{lemma}

\begin{proof}
See \cite[Lemma~2.1]{Fle16}.
\end{proof}

On the one hand the lemma shows that sequential weak*-to-weak continuity can be reformulated as a property of the range $\cR(A^\ast)$.
On the other hand, item (iii) in the lemma is obviously satisfied if $A$ has a bounded extension to some $\ell^q$-space with $q>1$.
Thus, restriction to sequentially weak*-to-weak continuous operators is a very weak restriction. But note that for $Y=\ell^1$ the identity mapping is a simple example of a not sequentially weak*-to-weak continuous operator.


First results on convergence rates for $\ell^1$-regularization if solutions are not sparse in \cite{BurFleHof13} were based on the assumption that
\begin{equation}\label{eq:range_inclusion}
e^{(k)}\in\cR(A^\ast)\qquad\text{for all }k\in\bbN.
\end{equation}
To our best knowledge such a condition appeared first for nonlinear operators in \cite{Gra09}.
In the present paper we show that the slightly weaker condition
\begin{equation}\label{eq:range_inclusion_closure}
e^{(k)}\in\overline{\cR(A^\ast)}\qquad\text{for all }k\in\bbN
\end{equation}
automatically holds for all injective, sequentially weak*-to-weak continuous operators and suffices, in combination with sequential weak*-to-weak continuity, to obtain convergence rates.
Here and in the whole paper an overlined subset of $\ell^\infty$ denotes the closure of this set with respect to the $\ell^\infty$-norm.

In \cite[Proposition~2.4]{BurFleHof13} it has been observed that $\overline{\cR(A^\ast)}=c_0$ if $\|Ae^{(k)}\|_Y$ converges to zero.
The same equality will be obtained in the present paper, but under the weaker assumption that $A$ is weak*-to-weak continuous.

Typically one has a decomposition
\[
A=\tilde{A}\circ L
\]
where $\tilde{A}:\tilde{X}\rightarrow Y$ maps from some Banach space $\tilde{X}$ into $Y$ and $L:\ell^1\rightarrow\tilde{X}$ is a synthesis operator with respect to some Schauder basis $(v_k)_{k\in\bbN}$, that is,
\[
Lx:=\sum_{k=1}^\infty x_k v_k.
\]
In practice $\tilde{X}$ is often a Hilbert space and $(v_k)_{k\in\bbN}$ is an orthonormal basis.
Then it is easy to see that \eqref{eq:range_inclusion} holds if and only if $v_k\in\cR(\tilde{A}^\ast)$ for all $k$.
If $\tilde{X}=\ell^p$ a similar result can be obtained and for a general Banach space $\tilde{X}$ one has to switch to biorthogonal systems in $\tilde{X}$ and $\tilde{X}^\ast$.

It had not been clear whether condition \eqref{eq:range_inclusion} holds for larger classes of operators.
An affirmative was given in \cite{AnzHofRam13} in the case that $\tilde{X}$ is a Hilbert space with a separable and dense linear subspace $V$, such that $(V,\tilde{X},V^\ast)$ forms a Gelfand triple. Among other examples it was shown that the Radon transform possesses this property, showing that a large class of in particular practical problems fulfill \eqref{eq:range_inclusion}.

On the other hand, a negative example was constructed in \cite{FleHeg14} showing that already a certain bidiagonal operator does not fulfill \eqref{eq:range_inclusion}. In order to overcome this deficiency, a weaker assumption for obtaining convergence rates was introduced which in principle states that there are  elements $\eta\in Y^\ast$ such that each basis element $e^{(k)}$ can be approximated via $A^\ast\eta$ with $[A^\ast\eta]_l=e^{(k)}_l$ for $l\leq k$ and $[A^\ast\eta]_l$ sufficiently small for all $l>k$.
For proving our main theorem we shall employ a very similar condition which in turn is a variant of an assumption that was used in \cite{FleHofVes16}, implying the one in \cite{FleHeg14}. In order to formulate it we introduce the projectors
\begin{equation}
P_n:\ell^\infty\rightarrow \ell^\infty,\quad P_nx:=(x_1,\ldots,x_n,0,\ldots)
\end{equation}
as the cut-off after the $n$-th entry.

\begin{condition}\label{ass2}
There exist a real sequence $(\gamma_n)_{n\in\bbN}$ and a constant $\mu\in[0,1)$ such that for each $n\in\bbN$ and each $\xi\in\ell^\infty$ with
\begin{equation*}
\xi_k\begin{cases}\in\{-1,0,1\},&\text{if }k\leq n,\\=0,&\text{if }k>n\end{cases}
\end{equation*}
there exists some $\eta=\eta(\mu,n,\xi)$ in $Y^\ast$ such that
\begin{itemize}
\item[(i)]
$P_nA^\ast\eta=\xi$,
\item[(ii)]
$\vert[(I-P_n)A^\ast\eta]_k\vert\leq\mu$\quad for all $k>n$,
\item[(iii)]
$\|\eta\|_{Y^\ast}\leq\gamma_n$.
\end{itemize}
\end{condition}

The existence of $\eta=\eta(\xi)$ in the condition depends heavily on the constants $\mu$ and $\gamma_n$. The interplay of these two constants seems to be rather complicated, but the proof of the main theorem will show that fixing some $\mu\in(0,1)$ one always finds a sequence $(\gamma_n)_{n\in\bbN}$ such that the condition holds. The growth of the $\gamma_n$ has influence on the convergence rate of $\ell^1$-regularization and it is not clear how to choose $\mu$ to make this growth as slow as possible.

\section{Proof part I: variational source condition}\label{sc:proof}

To prove Theorem~\ref{th:main} we start with the properties of $\varphi$ defined by \eqref{eq:phi}.
As an infimum of affine functions it is concave and upper semi-continuous. Concavity implies continuity on $(0,\infty)$ and from $\varphi(0)=0$, non-negativity and upper semi-continuity we obtain continuity of $\varphi$ on $[0,\infty)$. Monotonicity of $\varphi$ follows from monotonicity of $(\gamma_n)_{n\in\bbN}$.

The major part of the proof of Theorem~\ref{th:main} is to show that Condition~\ref{ass2} holds. Then we can refer to \cite[Theorem~2.2]{FleHofVes16} to obtain a variational source condition (or variational inequality)
\begin{equation}\label{eq:vsc}
\beta\|x-x^\dagger\|_{\ell^1}\leq\|x\|_{\ell^1}-\|x^\dagger\|_{\ell^1}+\varphi\bigl(\|Ax-Ax^\dagger\|_Y\bigr)\quad\text{for all }x\in\ell^1.
\end{equation}
with some constant $\beta\in(0,1]$.
This variational source condition is known to imply the asserted convergence rate, see \cite{Fle12,HofMat12}.

We now prove validity of Condition~\ref{ass2} and derive a variational source condition.
In the next section, for the reader's convenience, we present the remaining steps to obtain the rates result in our notation.

\begin{definition}
Let $X$ be some Banach space and let $U$ and $V$ be subspaces of $X$ and $X^\ast$, respectively.
The annihilator of $U\subset X$ in $X^\ast$ is
\begin{equation*}
U^\perp:=\{\xi\in X^\ast:\la\xi,x\ra_{X^\ast\times X}=0\text{ for all $x\in U$}\}
\end{equation*}
and the annihilator of $V\subset X^\ast$ in $X$ is
\begin{equation*}
V_\perp:=\{x\in X:\la\xi,x\ra_{X^\ast\times X}=0\text{ for all $\xi\in V$}\}.
\end{equation*}
\end{definition}


With the help of annihilators we can carry over well known relations between null spaces and ranges of $A$ and $A^\ast$ in Hilbert spaces to Banach spaces. We need the following relation.

\begin{lemma}\label{th:rana}
Denoting by $A^{\ast\ast}:=(A^\ast)^\ast:(\ell^\infty)^\ast\rightarrow Y^{\ast\ast}$ we have
\begin{equation*}
\overline{\cR(A^\ast)}=\cN(A^{\ast\ast})_\perp.
\end{equation*}
\end{lemma}

\begin{proof}
See, e.\,g., \cite[Lemma~3.1.16 and Proposition~1.10.15(c)]{Meg98}.
\end{proof}

To exploit the lemma we need information about the structure of $(\ell^\infty)^\ast$ which is not the same as $\ell^1$ but a strictly larger space, because $\ell^1$ is not reflexive. We have the following very useful characterization of $(\ell^\infty)^\ast$, which is a special case of \cite[Theorem~2.14]{Tak02}.

\begin{lemma}\label{th:l1c0}
Each element of $(\ell^\infty)^\ast$ is the sum of an element of $\ell^1$ and an element of $c_0^\perp$, that is,
\begin{equation*}
(\ell^\infty)^\ast=\ell^1\oplus c_0^\perp.
\end{equation*}
\end{lemma}

\begin{proof}
Let $u\in(\ell^\infty)^\ast$. Set
\begin{equation*}
x_k:=\la u,e^{(k)}\ra_{(\ell^\infty)^\ast\times\ell^\infty}.
\end{equation*}
Then $x=(x_k)_{k\in\bbN}\in\ell^1$ because
\begin{align*}
\sum_{k=1}^n\vert x_k\vert
&=\sum_{k=1}^n(\sgn x_k)\la u,e^{(k)}\ra_{(\ell^\infty)^\ast\times\ell^\infty}
=\left\la u,\sum_{k=1}^n(\sgn x_k)\,e^{(k)}\right\ra_{(\ell^\infty)^\ast\times\ell^\infty}\\
&\leq\|u\|_{(\ell^\infty)^\ast}.
\end{align*}
It remains to show $u-x\in c_0^\perp$. Indeed, for each $\xi\in c_0$ we have
\begin{align*}
\la u-x,\xi\ra_{(\ell^\infty)^\ast\times\ell^\infty}
&=\lim_{n\to\infty}\left\la u-x,\sum_{k=1}^n\xi_k\,e^{(k)}\right\ra_{(\ell^\infty)^\ast\times\ell^\infty}\\
&=\lim_{n\to\infty}\sum_{k=1}^n\left(\xi_k\,\la u,e^{(k)}\ra_{(\ell^\infty)^\ast\times\ell^\infty}-\xi_k\,\la x,e^{(k)}\ra_{\ell^1\times\ell^\infty}\right)\\
&=0.
\end{align*}
\end{proof}

Combining this result with Lemma \ref{th:rana} yields a full characterization of $\overline{\cR(A^\ast)}$.

\begin{proposition}\label{th:radense}
Let $A$ be injective and sequentially weak*-to-weak continuous. Then
\begin{equation*}
\overline{\cR(A^\ast)}=c_0.
\end{equation*}
\end{proposition}

\begin{proof}
From Lemma~\ref{th:l1c0} we we know that $A^{\ast\ast}$ maps $\ell^1\oplus c_0^\perp$ into $Y^{\ast\ast}$.
On the one hand, for each $x\in\ell^1$ and each $\eta\in Y^\ast$ we see
\begin{align*}
\la A^{\ast\ast}\,x,\eta\ra_{Y^{\ast\ast}\times Y^\ast}
&=\la x,A^\ast\,\eta\ra_{(\ell^\infty)^\ast\times\ell^\infty}
=\la x,A^\ast\,\eta\ra_{\ell^1\times\ell^\infty}
=\la A\,x,\eta\ra_{Y\times Y^\ast}\\
&=\la A\,x,\eta\ra_{Y^{\ast\ast}\times Y^\ast},
\end{align*}
that is, $A^{\ast\ast}\vert_{\ell^1}=A$. If $A$ is injective we have $\cN(A|_{\ell^1})=\{0\}$.
On the other hand, for each $u\in c_0^\perp$ and each $\eta\in Y^\ast$ we see
\begin{equation*}
\la A^{\ast\ast}\,u,\eta\ra_{Y^{\ast\ast}\times Y^\ast}
=\la u,A^\ast\,\eta\ra_{(\ell^\infty)^\ast\times\ell^\infty}
=0
\end{equation*}
because $A^\ast\,\eta\in\cR(A^\ast)\subseteq c_0$ as a consequence of weak*-to-weak continuity (cf.\ \cite[Lemma~2.1]{Fle16}).
Thus, $A^{\ast\ast}\vert_{c_0^\perp}=0$ and together with $\cN(A|_{\ell^1})=\{0\}$ it is $\cN(A^{\ast\ast})=c_0^\perp$.
\par
With Lemma~\ref{th:rana} we now obtain
\begin{equation*}
\overline{\cR(A^\ast)}=\cN(A^{\ast\ast})_\perp=(c_0^\perp)_\perp=c_0,
\end{equation*}
where the last equality is a consequence of the Hahn-Banach Theorem (cf.\ \cite[Proposition~1.10.15]{Meg98}).
\end{proof}

The converse, that $\overline{\cR(A^\ast)}=c_0$ implies injectivity, is in general not true, because one can show that $A$ is injective if and only if the weak*-closure of $\cR(A^\ast)$ coincides with $\ell^\infty$ (see \cite[Theorem~3.1.17(a)]{Meg98}).

Verification of Condition~\ref{ass2} will be completed by a corollary of the following proposition.

\begin{proposition}\label{th:ra}
Let $\varepsilon>0$ and let $n\in\bbN$. Then for each $\xi\in c_0$ there exists $\tilde{\xi}\in\cR(A^\ast)$ such that
\begin{equation*}
\tilde{\xi}_k=\xi_k\quad\text{for $k\leq n$}\qquad\text{and}\qquad\vert\tilde{\xi}_k-\xi_k\vert\leq\varepsilon\quad\text{for $k>n$}.
\end{equation*}
\end{proposition}

\begin{proof}
We proof the proposition by induction with respect to $n$. For $\xi\in c_0$ set
\begin{equation*}
\xi^+:=(\xi_1+\varepsilon,\xi_2,\xi_3,\ldots)\qquad\text{and}\qquad\xi^-:=(\xi_1-\varepsilon,\xi_2,\xi_3,\ldots).
\end{equation*}
By Proposition~\ref{th:radense} we find $\tilde{\xi}^+\in\cR(A^\ast)$ and $\tilde{\xi}^-\in\cR(A^\ast)$ with
\begin{equation*}
\|\tilde{\xi}^+-\xi^+\|_{\ell^\infty}\leq\varepsilon\qquad\text{and}\qquad\|\tilde{\xi}^--\xi^-\|_{\ell^\infty}\leq\varepsilon.
\end{equation*}
Consequently, $\tilde{\xi}^+_1\geq\xi_1\geq\tilde{\xi}^-_1$ and $\vert\tilde{\xi}^+_k-\xi_k\vert\leq\varepsilon$ as well as $\vert\tilde{\xi}^-_k-\xi_k\vert\leq\varepsilon$ for $k>1$. Thus we find a convex combination $\tilde{\xi}$ of $\tilde{\xi}^+$ and $\tilde{\xi}^-$ such that $\tilde{\xi}_1=\xi_1$. This $\tilde{\xi}$ obviously also satisfies $\vert\tilde{\xi}_k-\xi_k\vert\leq\varepsilon$ for $k>1$, which proves the proposition for $n=1$.
\par
Now let the proposition be true for $n=m$. We prove it for $n=m+1$.
Let $\xi\in c_0$ and set
\begin{align*}
\xi^+&:=(\xi_1,\ldots,\xi_m,\xi_{m+1}+\varepsilon,\xi_{m+2},\xi_{m+3},\ldots),\\
\xi^-&:=(\xi_1,\ldots,\xi_m,\xi_{m+1}-\varepsilon,\xi_{m+2},\xi_{m+3},\ldots).
\end{align*}
By the induction hypothesis we find $\tilde{\xi}^+\in\cR(A^\ast)$ and $\tilde{\xi}^-\in\cR(A^\ast)$ with 
\begin{equation*}
\tilde{\xi}^+_k=\xi_k=\tilde{\xi}^-_k\quad\text{for $k\leq m$}
\end{equation*}
and
\begin{equation*}
\vert\tilde{\xi}^+_k-\xi^+_k\vert\leq\varepsilon\quad\text{and}\quad\vert\tilde{\xi}^-_k-\xi^-_k\vert\leq\varepsilon\quad\text{for $k>m$}.
\end{equation*}
Consequently, $\tilde{\xi}^+_{m+1}\geq\xi_{m+1}\geq\tilde{\xi}^-_{m+1}$ and $\vert\tilde{\xi}^+_k-\xi_k\vert\leq\varepsilon$ as well as $\vert\tilde{\xi}^-_k-\xi_k\vert\leq\varepsilon$ for $k>m+1$. Thus we find a convex combination $\tilde{\xi}$ of $\tilde{\xi}^+$ and $\tilde{\xi}^-$ such that $\tilde{\xi}_{m+1}=\xi_{m+1}$. This $\tilde{\xi}$ obviously also satisfies $\tilde{\xi}_k=\xi_k$ for $k<m+1$ and $\vert\tilde{\xi}_k-\xi_k\vert\leq\varepsilon$ for $k>m+1$, which proves the proposition for $n=m+1$.
\end{proof}

\begin{corollary}\label{th:cond}
Let $A$ be injective and sequentially weak*-to-weak continuous. Then for each $\mu\in(0,1)$ there is a sequence $(\gamma_n)_{n\in\bbN}$ such that Condition~\ref{ass2} is satisfied.
\end{corollary}

\begin{proof}
Fix $\mu$ and $n$ and take some $\xi$ as described in Condition~\ref{ass2}.
By Proposition~\ref{th:ra} with $\varepsilon:=\mu$ there exists some $\eta$ such that $A^\ast\eta$ ($=\tilde{\xi}$ in the proposition) satisfies items (i) and (ii) in the condition.
\par
The set of all $\xi$ to be considered in Condition~\ref{ass2} for fixed $n$ is a bounded subset of a finite-dimensional subspace of $\ell^\infty$ and the set of corresponding $\eta=\eta(\xi)$ is contained in the preimage of this finite-dimensional subset with respect to the linear mapping $A^\ast$. Thus, the set of $\eta$ is bounded, too. Choosing $\gamma_n$ to be this bound we automatically satisfy item $(iii)$ in the condition.
\end{proof}

To obtain a variational source condition \eqref{eq:vsc} from Condition~\ref{ass2}, which is always satisfied under our standing assumptions, we use the estimates from \cite[proof of Theorem~2.2]{FleHofVes16}.

\begin{corollary}
Let $A$ be injective and sequentially weak*-to-weak continuous and let $\mu$ and $(\gamma_n)_{n\in\bbN}$ be as in Condition~\ref{ass2}. Then a variational source condition \eqref{eq:vsc} with $\beta=\frac{1-\mu}{1+\mu}$ and $\varphi$ given by \eqref{eq:phi} is fulfilled.
\end{corollary}

\begin{proof}
Fix $n\in\bbN$ and $x\in\ell^1$ and let $\xi:=\sgn P_n(x-x^\dagger)\in\ell^\infty$ be the sequence of signs of $P_n(x-x^\dagger)$.
Then by Condition~\ref{ass2} there is some $\eta$ such that
\begin{align*}
\lefteqn{\|P_n(x-x^\dag)\|_{\ell^1}
=\langle \xi, x-x^\dag\rangle_{\ell^\infty\times\ell^1}= \langle P_n A^\ast \eta, x-x^\dag\rangle_{\ell^\infty\times\ell^1}}\\
&\qquad=\langle P_n A^\ast \eta-A^\ast\eta, x-x^\dag\rangle_{\ell^\infty\times\ell^1}+\langle  A^\ast \eta, x-x^\dag\rangle_{\ell^\infty\times\ell^1}\\
&\qquad=-\langle (I-P_n)A^\ast\eta, (I-P_n)(x-x^\dag)\rangle_{\ell^\infty\times\ell^1}+\langle  A^\ast \eta, x-x^\dag\rangle_{\ell^\infty\times\ell^1}\\
&\qquad\leq \mu\|(I-P_n)(x-x^\dag)\|_{\ell^1}+\gamma_n\|Ax-Ax^\dag\|_Y.
\end{align*}
The triangle inequality yields
\begin{equation}\label{eq:pm}
\|P_n(x-x^\dagger)\|_{\ell^1}\leq \mu\bigl(\|(I-P_n)x\|_{\ell^1}+\|(I-P_n)x^\dagger\|_{\ell^1}\bigr)+\gamma_n\|Ax-Ax^\dagger\|_Y.
\end{equation}
Now
\begin{align*}
\lefteqn{\beta\|x-x^\dagger\|_{\ell^1}-\|x\|_{\ell^1}+\|x^\dagger\|_{\ell^1}}\\
&=\beta\|P_n(x-x^\dagger)\|_{\ell^1}+\beta\|(I-P_n)(x-x^\dagger)\|_{\ell^1}-\|P_nx\|_{\ell^1}-\|(I-P_n)x\|_{\ell^1}\\
&\quad+\|P_nx^\dagger\|_{\ell^1}+\|(I-P_n)x^\dagger\|_{\ell^1}
\end{align*}
together with
\begin{equation*}
\beta\|(I-P_n)(x-x^\dagger)\|_{\ell^1}\leq\beta\|(I-P_n)x\|_{\ell^1}+\beta\|(I-P_n)x^\dagger\|_{\ell^1}
\end{equation*}
and
\begin{equation*}
\|P_nx^\dagger\|_{\ell^1}=\|P_n(x-x^\dagger-x)\|_{\ell^1}\leq\|P_n(x-x^\dagger)\|_{\ell^1}+\|P_nx\|_{\ell^1}
\end{equation*}
shows
\begin{align*}
\lefteqn{\beta\|x-x^\dagger\|_{\ell^1}-\|x\|_{\ell^1}+\|x^\dagger\|_{\ell^1}}\\
&\qquad\leq 2\|(I-P_n)x^\dagger\|_{\ell^1}+(1+\beta)\|P_n(x-x^\dagger)\|_{\ell^1}\\
&\qquad\quad\,-(1-\beta)\bigl(\|(I-P_n)x\|_{\ell^1}+\|(I-P_n)x^\dagger\|_{\ell^1}\bigr).
\end{align*}
Combining this estimate with the previous estimate \eqref{eq:pm} and taking into account that $\beta=\frac{1-\mu}{1+\mu}$
and $\mu=\frac{1-\beta}{1+\beta}$ we obtain
\begin{align*}
\beta\|x-x^\dagger\|_{\ell^1}-\|x\|_{\ell^1}+\|x^\dagger\|_{\ell^1}
&\leq 2\|(I-P_n)x^\dagger\|_{\ell^1}+\frac{2}{1+\mu}\,\gamma_n\|Ax-Ax^\dagger\|_Y\\
&\leq 2\|(I-P_n)x^\dagger\|_{\ell^1}+2\gamma_n\|Ax-Ax^\dagger\|_Y.
\end{align*}
Taking the infimum over all $n\in\bbN$ completes the proof.
\end{proof}

Now that we arrived at a variational source condition we summarize some observations the subtle observer can make en route. 

\begin{proposition}\label{thm:main}
If $A$ is sequentially weak*-to-weak continuous, the following statements are equivalent:
\begin{itemize}
\item[(i)]
Condition~\ref{ass2} holds for each $\mu\in(0,1)$,
\item[(ii)]
$e^{(k)}\in\overline{\cR(A^\ast)}$ for all $k\in\N$,
\item[(iii)]
$\overline{\cR(A^\ast)}=c_0$,
\item[(iv)]
$A$ is injective.
\end{itemize}
\end{proposition}

\begin{proof}
We show (i)$\Rightarrow$(ii)$\Rightarrow$(iii)$\Rightarrow$(iv)$\Rightarrow$(i).
\par
(i)$\Rightarrow$(ii): Fix $k$, fix $n\geq k$, take a sequence $(\mu_m)_{m\in\bbN}$ in $(0,1)$ with $\mu_m\to 0$ and choose $\xi:=e^{(k)}$ in Condition~\ref{ass2}.
Then for a corresponding sequence $(\eta_m)_{m\in\bbN}$ from Condition~\ref{ass2} we obtain
\begin{equation*}
\|e^{(k)}-A^\ast\eta_m\|_{\ell^\infty}
\leq\|e^{(k)}-P_nA^\ast\eta_m\|_{\ell^\infty}+\|(I-P_n)A^\ast\eta_m\|_{\ell^\infty}.
\end{equation*}
The first summand is zero by the choice of $\xi$ and the second summand is bounded by $\mu_m$. Thus, $\|e^{(k)}-A^\ast\eta_m\|_{\ell^\infty}\to 0$ if $m\to\infty$.
\par
(ii)$\Rightarrow$(iii):
$(e^{(k)})_{k\in\bbN}$ is a Schauder basis in $c_0$. Thus, $c_0\subseteq\overline{\cR(A^\ast)}$. In \cite[Lemma~2.1]{Fle16} we find that weak*-to-weak continuity implies $\cR(A^\ast)\subseteq c_0$ and hence also $\overline{\cR(A^\ast)}\subseteq c_0$.
\par
(iii)$\Rightarrow$(iv):
One easily shows that $\overline{\cR(A^\ast)}\subseteq\cN(A)^\perp$. Thus, $c_0\subseteq\cN(A)^\perp$. If we have some $x\in\ell^1$ with $Ax=0$, then for each $u\in c_0\subseteq\cN(A)^\perp$ we obtain
\begin{equation*}
\la x,u\ra_{\ell^1\times c_0}=\la u,x\ra_{\ell^\infty\times\ell^1}=0,
\end{equation*}
which is equivalent to $x=0$.
\par
(iv)$\Rightarrow$(i): See Corollary~\ref{th:cond}.
\end{proof}

Note that we have \eqref{eq:range_inclusion} if and only if Condition~\ref{ass2} holds with $\mu=0$. According to Proposition~\ref{thm:main} one might start with \eqref{eq:range_inclusion_closure} instead of Condition~\ref{ass2}. This is an obvious generalization of \eqref{eq:range_inclusion}.
While it is in general not easy to decide whether or not \eqref{eq:range_inclusion} holds, \eqref{eq:range_inclusion_closure} or, equivalently, injectivity of the operator $A$ can be verified easily.
The fulfillment of sequential weak*-to-weak continuity follows in practically all relevant cases from the construction of the problem, see the discussion in the first lines of Section~\ref{sc:basis}.

\section{Proof part II: convergence rates}\label{sc:proof2}

We now collect further proof pieces from the literature to provide the reader with a full proof of Theorem~\ref{th:main}. The missing part of the proof is the step from a variational source condition~\eqref{eq:vsc} to the error estimate in Theorem~\ref{th:main}.
For a priori chosen regularization parameter we follow the proof of \cite[Theorem~1]{HofMat12} (while improving constants slightly) and in case of the discrepancy principle we follow the arguments in \cite{Fle12}.

In the two proofs we exploit the properties of the function $\varphi$ in Theorem~\ref{th:main} several times. Simple calculations show that $t\mapsto\frac{\varphi(t)}{t}$ is decreasing. As a consequence we see that $\varphi(ct)\leq c\varphi(t)$ if $c\geq 1$. Both observations will be used without further notice.

\begin{proposition}
Let the variational source condition \eqref{eq:vsc} be satisfied and choose $\alpha$ in \eqref{eq:tikh} such that
\begin{equation*}
c_1\frac{\delta^p}{\varphi(\delta)}\leq\alpha\leq c_2\frac{\delta^p}{\varphi(\delta)}
\end{equation*}
with constants $c_1,c_2>0$.
Then
\begin{equation*}
\|x_\alpha^\delta-x^\dagger\|_{\ell^1}\leq\frac{1}{\beta}\left(1+\frac{1}{c_1}+(1+2c_2)^{\frac{1}{p-1}}\right)\varphi(\delta)\qquad\text{for all }\delta>0.
\end{equation*}
\end{proposition}

\begin{proof}
Because $x_\alpha^\delta$ is a minimizer of \eqref{eq:tikh} we have
\begin{align*}
\|x_\alpha^\delta\|_{\ell^1}-\|x^\dagger\|_{\ell^1}
&=\frac{1}{\alpha}\bigl(T_\alpha^\delta(x_\alpha^\delta)-\alpha\|x^\dagger\|_{\ell^1}-\|Ax_\alpha^\delta-y^\delta\|_Y^p\bigr)\\
&\leq\frac{1}{\alpha}\bigl(\|Ax^\dagger-y^\delta\|_Y^p-\|Ax_\alpha^\delta-y^\delta\|_Y^p\bigr)\\
&\leq\frac{1}{\alpha}\bigl(\delta^p-\|Ax_\alpha^\delta-y^\delta\|_Y^p\bigr).
\end{align*}
and thus the variational source condition \eqref{eq:vsc} implies
\begin{equation}\label{eq:proof1}
\beta\|x_\alpha^\delta-x^\dagger\|_{\ell^1}
\leq\frac{1}{\alpha}\bigl(\delta^p-\|Ax_\alpha^\delta-y^\delta\|_Y^p\bigr)+\varphi\bigl(\|Ax_\alpha^\delta-Ax^\dagger\|_Y\bigr).
\end{equation}
Because $\beta\|x_\alpha^\delta-x^\dagger\|_{\ell^1}\geq 0$, we obtain
\begin{equation*}
\|Ax_\alpha^\delta-y^\delta\|_Y^p\leq\delta^p+\alpha\varphi\bigl(\|Ax_\alpha^\delta-Ax^\dagger\|_Y\bigr).
\end{equation*}
If $\|Ax_\alpha^\delta-y^\delta\|_Y\leq\delta$, then the triangle inequality, the properties of $\varphi$ and the parameter choice imply
\begin{equation*}
\|Ax_\alpha^\delta-y^\delta\|_Y^p
\leq\delta^p+\alpha\varphi(2\delta)
\leq\delta^p+2\alpha\varphi(\delta)
\leq(1+2c_2)\delta^p,
\end{equation*}
that is,
\begin{equation*}
\|Ax_\alpha^\delta-y^\delta\|_Y\leq(1+2c_2)^{\frac{1}{p}}\delta\leq(1+2c_2)^{\frac{1}{p-1}}\delta.
\end{equation*}
If, on the other hand, $\|Ax_\alpha^\delta-y^\delta\|_Y>\delta$, then
\begin{align*}
\|Ax_\alpha^\delta-y^\delta\|_Y^p
&\leq\delta^p+\alpha\varphi\bigl(\|Ax_\alpha^\delta-y^\delta\|_Y+\delta\bigr)\\
&=\delta^p+\alpha\frac{\varphi\bigl(\|Ax_\alpha^\delta-y^\delta\|_Y+\delta\bigr)}{\|Ax_\alpha^\delta-y^\delta\|_Y+\delta}\bigl(\|Ax_\alpha^\delta-y^\delta\|_Y+\delta\bigr)\\
&\leq\delta^p+\alpha\frac{\varphi(\delta)}{\delta}\bigl(\|Ax_\alpha^\delta-y^\delta\|_Y+\delta\bigr)\\
&\leq\delta^{p-1}\|Ax_\alpha^\delta-y^\delta\|_Y+2\alpha\frac{\varphi(\delta)}{\delta}\|Ax_\alpha^\delta-y^\delta\|_Y
\end{align*}
and thus,
\begin{equation*}
\|Ax_\alpha^\delta-y^\delta\|_Y
\leq\left(\delta^{p-1}+2\alpha\frac{\varphi(\delta)}{\delta}\right)^{\frac{1}{p-1}}
\leq(1+2c_2)^{\frac{1}{p-1}}\delta.
\end{equation*}
In both cases \eqref{eq:proof1} can be further estimated to obtain
\begin{align*}
\beta\|x_\alpha^\delta-x^\dagger\|_{\ell^1}
&\leq\frac{1}{\alpha}\bigl(\delta^p-\|Ax_\alpha^\delta-y^\delta\|_Y^p\bigr)+\varphi\bigl(\|Ax_\alpha^\delta-y^\delta\|_Y+\delta\bigr)\\
&\leq\frac{\delta^p}{\alpha}+\varphi\left(\left(1+(1+2c_2)^{\frac{1}{p-1}}\right)\delta\right)\\
&\leq\frac{\delta^p}{\alpha}+\left(1+(1+2c_2)^{\frac{1}{p-1}}\right)\varphi(\delta)
\end{align*}
and the lower bound for $\alpha$ leads to
\begin{equation*}
\beta\|x_\alpha^\delta-x^\dagger\|_{\ell^1}
\leq\frac{\varphi(\delta)}{c_1}+\left(1+(1+2c_2)^{\frac{1}{p-1}}\right)\varphi(\delta).\qedhere
\end{equation*}
\end{proof}

Note that in the proof we used arguments similar to the ones in \cite{HofMat12}, but made changes in the details leading to a better constant in the obtained error estimate. Corresponding estimates in \cite[Theorem~1]{HofMat12} lead to
\begin{equation*}
\|x_\alpha^\delta-x^\dagger\|_{\ell^1}\leq\frac{1}{\beta}\left(1+2(2+p)^{\frac{1}{p-1}}\right)\varphi(\delta),
\end{equation*}
which has a greater constant factor than our estimate. Our estimate with the parameter choice from \cite{HofMat12}, that is $c_1=c_2=1$, reads
\begin{equation*}
\|x_\alpha^\delta-x^\dagger\|_{\ell^1}\leq\frac{1}{\beta}\left(2+3^{\frac{1}{p-1}}\right)\varphi(\delta).
\end{equation*}

\begin{proposition}
Let the variational source condition \eqref{eq:vsc} be satisfied and choose $\alpha$ in \eqref{eq:tikh} according to the discrepancy principle \eqref{eq:dp}.
Then
\begin{equation*}
\|x_\alpha^\delta-x^\dagger\|_{\ell^1}\leq \frac{1+\tau}{\beta}\varphi(\delta)\qquad\text{for all }\delta>0.
\end{equation*}
\end{proposition}

\begin{proof}
Because $x_\alpha^\delta$ is a minimizer of \eqref{eq:tikh} we have
\begin{align*}
\|x_\alpha^\delta\|_{\ell^1}-\|x^\dagger\|_{\ell^1}
&=\frac{1}{\alpha}\bigl(T_\alpha^\delta(x_\alpha^\delta)-\alpha\|x^\dagger\|_{\ell^1}-\|Ax_\alpha^\delta-y^\delta\|_Y^p\bigr)\\
&\leq\frac{1}{\alpha}\bigl(\|Ax^\dagger-y^\delta\|_Y^p-\|Ax_\alpha^\delta-y^\delta\|_Y^p\bigr)\\
&\leq\frac{1}{\alpha}\bigl(\delta^p-\|Ax_\alpha^\delta-y^\delta\|_Y^p\bigr).
\end{align*}
and taking into account the left-hand inequality in \eqref{eq:dp} we obtain
\begin{equation*}
\|x_\alpha^\delta\|_{\ell^1}-\|x^\dagger\|_{\ell^1}\leq 0.
\end{equation*}
The variational source condition \eqref{eq:vsc} thus implies
\begin{align*}
\beta\|x_\alpha^\delta-x^\dagger\|_{\ell^1}
&\leq\varphi\bigl(\|Ax_\alpha^\delta-Ax^\dagger\|_Y\bigr)
\leq\varphi\bigl(\|Ax_\alpha^\delta-y^\delta\|_Y+\delta\bigr)
\end{align*}
and the right-hand side in \eqref{eq:dp} yields
\begin{equation*}
\beta\|x_\alpha^\delta-x^\dagger\|_{\ell^1}
\leq\varphi\bigl((1+\tau)\delta\bigr)
\leq(1+\tau)\varphi(\delta).\qedhere
\end{equation*}
\end{proof}

\section{Remarks and open questions}

In the present paper we only consider the decay of the components of the solution $x^\dagger$ in their natural ordering $x^\dagger_1,x^\dagger_2,\ldots$. In \cite{FleHofVes16} a more general formulation was used. There the decay of the components after ordering them by size was considered, which may improve the error estimate for $\|x_\alpha^\delta-x^\dagger\|$ slightly. The same technique can be applied in the present paper, too, but to avoid notational intricateness we did not implement this feature.

Next to the considered a priori parameter choice and to the discrepancy principle also other parameter choice rules lead to the desired error estimate. For example the sequential discrepancy principle or the Lepski\u{\i} principle can be used (cf.\ \cite{HofMat12}).

An open question is whether the assumption that $A$ is injective can be dropped. The case of non-injective operators is of substantial interest in compressed sensing. Indeed, in \cite{Fle16} convergence rates for $\ell^1$-regularization were proven without the use of injectivity and also no finite basis injectivity or related properties were assumed. There, source-type conditions quite similar to Condition~\ref{ass2} were required and the question is whether those source-type conditions are always satisfied. Following the ideas of Section~\ref{sc:proof}, Proposition~\ref{th:radense} would state
\begin{equation*}
\overline{\cR(A^\ast)}=\cN(A)^\perp\cap c_0,
\end{equation*}
where $\cN(A)$ denotes the null space of $A$. But the proof of Proposition~\ref{th:ra} cannot be carried over directly to the non-injective case.
Perhaps, additional assumptions on the `angle' between $\cN(A)$ and the faces of the unit ball in $\ell^1$ are required.

Another open problem to be solved in future is the interplay between $\mu$ and $\gamma_n$ in Condition~\ref{ass2}. We know that there are situations which do not allow $\mu=0$ (cf.\ \cite{FleHeg14}), but on the other hand the condition holds for all $\mu\in(0,1)$ as we have shown. The $\gamma_n$ obviously depend on $\mu$ and we would like to know more about this dependence. In particular, we do not know whether $\mu$ influences the asymptotic behavior of the $\gamma_n$ if $n\to\infty$.

\section*{Acknowledgments}

We thank Bernd Hofmann (TU Chemnitz) for many valuable comments on a draft of this paper and for fruitful discussions on the subject.
Research was supported by DFG grants FL 832/1-2, HO 1454/8-2 and HO 1454/10-1.

\bibliography{l1without}

\end{document}